 \documentclass[envcountsect, envcountsame]{svmult}
 \usepackage{etex}

\smartqed

\usepackage{mathptmx}       
\usepackage{helvet}         
\usepackage{courier}        
\usepackage{type1cm}        
\usepackage{makeidx}         
\usepackage{graphicx}        
\usepackage{multicol}        
\usepackage[bottom]{footmisc}

\makeindex             

\usepackage{amstext,amsfonts,amssymb,amscd,amsbsy,amsmath}
\usepackage{mathtools}
\usepackage{tikz-cd}	
\usepackage{tikz}
\usetikzlibrary{calc}
\usepackage[percent]{overpic}

\newcommand{\PP}{\mathbb{P}}

\newcommand{\RR}{\mathbb{R}}

\DeclareMathOperator{\trop}{trop}

\usepackage{todonotes}
\usepackage{xcolor}
\usepackage[all]{xy}

\begin{document}

\title*{Tritangent planes to space sextics: the algebraic and tropical stories}
\titlerunning{Tritangents to sextic curves.} 
\author{Corey Harris and Yoav Len}
\institute{Corey Harris \at Florida State University, 208 Love Building, 1017 Academic Way, Tallahassee, FL 32306-4510, \email{charris@math.fsu.edu}
\and Yoav Len \at University of Waterloo, 200 University Ave W, Waterloo, ON N2L 3G1. \email{yoav.len@uwaterloo.ca.}}
%
%
\maketitle

\abstract*{We discuss the  classical problem of counting planes tangent to general canonical sextic curves at three points.  
We determine the number of real tritangents when such a curve is real.
We then revisit a curve constructed by Emch with the greatest known number of real tritangents, and conversely construct a curve with very few real tritangents. Using recent results on the relation between algebraic and tropical theta characteristics, we show that the tropicalization of a canonical sextic curve has $15$ tritangent planes.}

\abstract{We discuss the  classical problem of counting planes tangent to general canonical sextic curves at three points.  
We determine the number of real tritangents when such a curve is real.
We then revisit a curve constructed by Emch with the greatest known number of real tritangents, and conversely construct a curve with very few real tritangents. Using recent results on the relation between algebraic and tropical theta characteristics, we show that the tropicalization of a canonical sextic curve has $15$ tritangent planes.}

\section{Introduction}
\label{sec:intro}
In this paper, we study tritangent planes to general sextic curves in three dimensional projective space (which, in particular, are not hyperelliptic). By general, we mean that the curve is the intersection of a smooth quadric and a smooth cubic. 
A plane in space is determined by three parameters, and when chosen generically, it  meets a sextic in six points. Requiring that two contact points coincide to form a tangent imposes a single condition on the parameters. Therefore, a finite number of planes are expected to be tangent to the curve at three points. 
Making this argument precise and finding the exact number of tritangent planes is more subtle, and dates back to the mid-19th century
with the work of Clebsch \cite{Cle64}.
Sixty years later, an understanding of these tritangents was the impetus and principle goal for Coble in Algebraic Geometry and Theta Functions \cite{Cob61} 

The tritangent planes to a sextic are closely related with other classical problems such as the 27 lines on a cubic surface \cite[Chapter V.4]{Har77}  and the 28 bitangents to a quartic curve \cite[Chapter 6]{Dol12}.  The projection from a general point of a cubic surface is a double cover of a plane branched along a quartic.  The image of each of the 27 lines is bitangent to this quartic, with an additional bitangent given by blowing up the indeterminacy locus of the projection. 
Quite similarly, a del Pezzo surface of degree one forms a double cover of a quadric cone, branched along a smooth sextic of genus 4.
The $(-1)$-curves are mapped to conics, each meeting the sextic in three points, and the planes containing these conics are tritangent planes, i.e., they have intersection multiplicity two at each of the three points.

In a lecture given by Arnold at his 60th birthday conference at the Fields Institute, he referred to this as one of his mathematical trinities (\cite{Arn97}). Other examples of trinities are the exceptional Lie algebras $E_6,E_7,E_8$, the rings $\mathbb{R},\mathbb{C},\mathbb{H}$,  and the three polytopes tetrahedron, cube, and dodecahedron. 
\begin{quote}
The next dream I want to present is an even more fantastic set of theorems and conjectures. Here I also have no theory and actually the ideas form a kind of religion rather than mathematics.
The key observation is that in mathematics one encounters many trinities... I mean the existence of some “functorial” constructions connecting different trinities. The knowledge of the existence of these diagrams provides some new conjectures which might turn to be true theorems... I have heard from John MacKay that the  straight lines on a cubical surface, the  tangents of a quartic plane curve, and the  tritangent planes of a canonic sextic curve of genus 4 form a trinity parallel to $E_6$, $E_7$ and $E_8$.
\end{quote}

Our interest in this paper stems from two variations of the classical problem: the real case and the tropical case. 
In the case of a real space sextic, one may ask how many of the tritangent planes are real.  In Sect.~\ref{sec:real}, we appeal to the theory of real theta characteristics to show that the answer depends on the topology of the real curve: the number of connected components and how they are arranged on the Riemann surface of the complex curve.  In some cases, all the tritangents may be real, but their three points of tangency may include a complex-conjugate pair (Theorem~\ref{realTheta}).  We continue to explore this phenomenon through explicit examples.  We reexamine a construction of Arnold Emch, in which he claimed to find 120 tritangents, and show that he over counted.

\medskip

\noindent\textbf{Theorem~\ref{only108}.} \textit{Emch's curve has only 108 planes tritangent at 3 real points.}
\medskip 

\noindent On the other extreme, we construct a real space sextic with only one connected component and find its 8 real tritangent planes. 

In Sect.~\ref{sec:trop} we set up a tropical formulation of the problem. Once the notion of a tropical tritangent plane is established, it is natural to ask how many such planes are carried by a tropical sextic curve. 
This is a natural sequel to earlier tropical counting problems, such as the number of lines on a tropical cubic surface \cite{Vig10}, and the number of bitangents to a tropical plane quartic \cite{BLM16, CJ17}.

\medskip

\noindent\textbf{Theorem~\ref{thm:mainTropThm}.} \textit{A smooth tropical sextic curve $\Gamma$ in $\mathbb{R}^3$ has at most $15$ classes of tritangent planes. If it is the  tropicalization of a sextic $C$ on a smooth quadric in $\mathbb{P}^3$, then it has  exactly $15$ equivalence classes of tritangent planes.}

\medskip
\noindent In (Lemma~\ref{lemma:tropicalRuled}) we show that the question can actually be replaced by the simpler problem of counting  tritangents to tropical curves of bi-degree $(3,3)$ in tropical $\PP^1 \times \PP^1$. This result paves the way for a computational study of tropical tritangents.

\section{Algebraic space sextics}
\label{sec:alg}

Throughout this section, we work over an algebraically closed field $\Bbbk$ of characteristic different from $2$. For simplicity, the reader may assume that the  field is $\mathbb{C}$.

Let $C\subset\mathbb{P}^3$ be the intersection of a quadric   and a cubic surface. Such a curve is a smooth canonical sextic \cite[Chapter IV, Prop. 6.3]{Har77}. 
The intersection of a hyperplane with $C$ is a divisor of degree $6$ and rank $3$, and is the canonical divisor $K_C$ of $C$. In particular, the genus of $C$ is $4$. It follows that whenever $H$ is tangent to $C$ at $3$ points, those points form a divisor $D$ such that $2D\simeq K_C$.

\begin{definition}
A divisor class $[D]$ with the property that $2D\simeq K_C$ is called a \emph{theta characteristic}. A theta characteristic is said to be  \emph{odd} or \emph{even} according with the parity of $\dim H^0(C,D)$. 
\end{definition}


\begin{theorem}
\label{thm:tritangentsAlgebraic}
Let $C$ be the sextic obtained from the intersection of a smooth cubic and a smooth quadric in $\mathbb{P}^3$. Then it has $120$ tritangent planes, in bijection with its odd theta characteristics. 
\end{theorem}
\begin{proof}
Let $D$ be a theta characteristic of $C$ obtained from a hyperplane section, and set $h \coloneqq \dim H^0(C,D)$. We will show that $h=1$. To begin with,  Clifford's theorem \cite[Chapter 3.1]{ACGH85} implies that $h$ is strictly smaller than $3$. As $D$ is obtained from intersecting a curve with a plane, it is effective, so $h$ also cannot be $0$. To see that $h$ cannot be $2$, recall the geometric version of the Riemann--Roch theorem: in the canonical embedding, the support of a divisor of degree $d$ with $h$ global sections spans a subspace of dimension $d-h$. It follows that $h=2$ if and only if the contact points are co-linear. To see that this is impossible, let $Q$ be the quadric containing $C$. Since $Q$ is smooth, it is isomorphic to $\mathbb{P}^1\times\mathbb{P}^1$. The intersection of $Q$ with a plane $H$ is a $(1,1)$-curve on $Q$. Therefore, if it contains a line of one of the rulings, it also contains a line in the other. In particular, $H$ intersects $C$ at points not in the support of $D$ and is therefore not a tritangent. 

On the other hand, given an odd theta characteristic $p_1+p_2+p_3$, a plane $H$ through $p_1,p_2,p_3$ intersects $C$ at a divisor of the form $p_1+p_2+p_3+q_1+q_2+q_3$ for some points $q_1,q_2,q_3$. Since $p_1+p_2+p_3+q_1+q_2+q_3$ and $2(p_1+p_2+p_3)$ are both canonical, we get an equivalence of divisors $p_1+p_2+p_3\simeq q_1+q_2+q_3$. The rank of $p_1+p_2+p_3$ is zero by Clifford's theorem and the fact that it is an odd theta characteristic. It follows that these divisors are, in fact, equal. We conclude that 
$p_1+p_2+p_3+q_1+q_2+q_3=2p_1+2p_2+2p_3.$
In other words, $H$ is tritangent to $C$ at $p_1,p_2,p_3$.

We conclude that tritangent planes are in bijection with the odd theta characteristics of $C$. It is well known \cite{Mum71} that the number of odd characteristics of a curve of genus $g$ is $2^{g-1}(2^g-1)$, which is $120$ in this case.
\qed
\end{proof}

A canonical sextic does have two classes of co-linear 
divisors of degree $3$. Indeed, those correspond to intersections of the sextic with the rulings of the ambient quadratic surface. However, as seen in the proof above,  such a divisor is never obtained as the intersection of a hyperplane with the curve. 

\begin{remark}\label{rmk:oneOneCurves}
A smooth quadratic surface in $\mathbb{P}^3$ is isomorphic to  $\mathbb{P}^1\times\mathbb{P}^1$ via the Segre embedding \cite[Lemma 3.31]{Liu02}. Under this isomorphism, the sextic corresponds to a curve of bi-degree $(3,3)$ on $\mathbb{P}^1\times\mathbb{P}^1$, and a tritangent plane corresponds to a tritangent $(1,1)$-curve. 
It follows that a $(3,3)$-curve on a quadratic surface has $120$ tritangent $(1,1)$-lines as well. This result can also be deduced directly, similarly to the proof of Theorem~\ref{thm:tritangentsAlgebraic}.
\end{remark}


Our initial interest in tritangent planes came from studying tritangent planes to Bring's curve, as part of the apprenticeship workshop at the Fields Institute \cite[Problem~4 on Curves]{Stu17}. 
Bring's curve is a space sextic, traditionally written in supernumerary coordinates by considering a special plane in $\mathbb{P}^4$.  In particular, it is the intersection of the quadric given by $x^2+y^2+z^2+t^2+u^2=0$ and the cubic given by $x^3+y^3+z^3+t^3+u^3=0$ in the plane $x+y+z+t+u=0$.

Edge~\cite{Edg81} found equations for all 120 tritangent planes to Bring's curve, which  appear in two flavors.  Type (i) tritangent planes are determined by three \emph{stalls} of the curve.  These are the points at which the osculating plane has order of contact higher than expected.  In this case, the general point on the curve has order 3 contact with the osculating plane, and the stall points have order 4.  Pl\"ucker's formulas for space curves tell us that there are exactly 60 stalls on Bring's curve.

Let $\alpha,\beta,\gamma$ be the three distinct roots of $\theta^3 + 2\theta^2 + 3\theta + 4=0$.  One can check that each of the equations $\gamma t - \beta u = 0$, $\alpha u - \gamma z = 0$, $\beta z - \alpha t = 0$
defines a plane which is tritangent to our curve and contains the tangent line at the stall point $(1,1,\alpha,\beta,\gamma)$.  The rest of the type (i) planes are given by replacing $(z,t,u)$ with any of the $60$ ordered triples in $\{x,y,z,t,u\}$. 
 Each of these contains the tangent line at a stall point given by an appropriate permutation of the coordinates of $[1:1:\alpha:\beta:\gamma]$. The construction yields three tritangent planes through each of the 60 stall points of the curve, and every plane contains three  stall points.  In other words, the containment relation between type (i) tritangents and stalls determines the edges of a bipartite graph $B(m,60)$ such that every vertex has valence $3$.  Thus $m=60$, so we see that there are $60$ such tritangent planes.
 

The type (ii) tritangent planes each contain exactly one stall point.  One of them is given by
\begin{equation*}
(\alpha -1)(\alpha +4)z + (\beta-1)(\beta+4)t + (\gamma-1)(\gamma+4)u = 0,
\end{equation*}
and the rest are obtained, again, by replacing $(z,t,u)$ with ordered triples in $\{x,y,z,t,u\}$.
This is summarized in the theorem.

\begin{theorem}[\cite{Edg81}]
Bring's curve has $60$ tritangent planes of type (i), and $60$ tritangent planes of type (ii) with equations as above.
\qed
\end{theorem}

\section{Tritangents to real space curves}
\label{sec:real}

In this section, we restrict ourselves to smooth curves that are defined over the real numbers with non-empty real part. As is known, the real part of a curve consists of a disjoint union of ovals, where by \emph{oval} we mean a simple closed loop. For a curve of genus $g$, the number of these ovals cannot exceed $g+1$. See \cite{Vir08,BCR98} for a nice introduction to real algebraic geometry.

We say that a tritangent plane is  \emph{real}  if it is defined over the reals, and \emph{totally-real} if in addition the tangency points are all real. For a real tritangent that is not totally-real, the tangency points consist of a real point and a pair of complex conjugate points. 
For a smooth sextic on a smooth quadric in $\mathbb{P}^3$,  real tritangent planes are in bijection with real odd theta characteristics.
Their number is given as follows. 

\begin{proposition}[\cite{Kra96}]\label{realTheta}
Let $C$ be a real curve of genus $g$, and assume that its real part $C(\mathbb{R})$ consists of $s> 0$ ovals. 
\begin{enumerate}
\item If $C(\mathbb{R})$ separates $C$, then there are $2^{g-1}(2^{s-1}+1)$ real even theta characteristics and $2^{g-1}(2^{s-1}-1)$ real odd ones.
\item If $C(\mathbb{R})$ doesn't separate $C$, then there are $2^{g+s-2}$ real even and $2^{g+s-2}$ real odd theta characteristics.
\end{enumerate}
\end{proposition}


A real curve of genus $g$ with exactly $g+1$ ovals is referred to as an \emph{M-curve}. Any disjoint union of $g+1$ cycles on a Riemann surface of genus $g$  separates the surface, so  an M-curve always corresponds to the first case of Proposition~\ref{realTheta}. In particular, a canonical space sextic with $5$ ovals has $120$ real tritangent planes.  The question remains, how many of them are totally-real?

In \cite{Emc28}, Emch claimed that the tritangents of a real space sextic with 5 ovals are all totally-real, and constructed an example of such a curve and its tritangents. However, as we will see below, several of the planes were over counted, and only $108$ of its tritangents are totally-real.  We are not aware of any previous literature that has addressed this issue.

We begin by considering a union of three lines in $\PP^2$ so that they bound an equilateral triangle with incenter at the origin.  For instance, if we choose one line to be of the form $x=a$, we find that the other lines should have slope $\frac{\pm 1}{\sqrt 3}$.  Choosing $a=-\sqrt{3}$ yields
$p(x,y) = (x+\sqrt{3})(x-y\sqrt{3}-3)(x+y\sqrt{3}-3)$
and $V(p(x,y)) \subset \mathbb{A}^2$ is our union of lines. 

\begin{figure}
\begin{minipage}{.5\textwidth}
\centering
\includegraphics[width=0.8\linewidth]{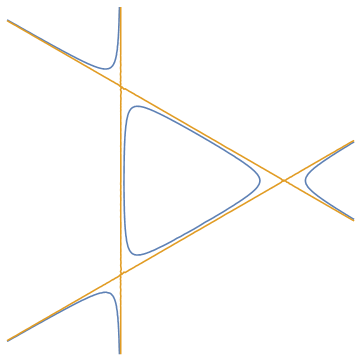}
\caption{Union of lines and a smooth cubic.}
\label{fig:cubicLines}
\end{minipage}
\begin{minipage}{.5\textwidth}
\centering
\begin{overpic}[width=0.77\textwidth,trim=30mm 32mm 30mm 23mm,clip]{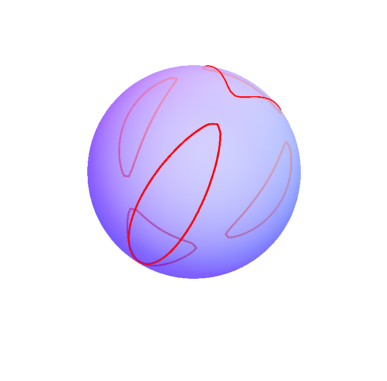}  
  \put (40,50) {\large$O_1$}
  \put (80,40) {$O_2$}
  \put (20,70){$O_3$}
  \put (70,80) {$N$}
  \put (50,20) {\small$S$}
\end{overpic}
\caption[Fig5ovals]{Real sextic with five ovals.}
\label{overpicfig}
\end{minipage}
\end{figure}


The set of points $\{(x,y) \;|\; p(x,y)=2\}$ is a smooth cubic curve with four real branches (one of which is an oval bounded by our triangle).  The polynomial $c(x,y,z):=p(x,y)-2$ has zeros along a cubic cylinder in $\mathbb{A}^3$.  A sphere centered at the origin of sufficiently large radius meets each of the components of the cubic, and it meets the central component twice.  Therefore, the intersection of the cubic surface with a sphere yields a space sextic with five ovals.
We refer to the top and bottom ovals as $N$ and $S$, and the other three as $O_1,O_2,O_3$.

\begin{theorem}\label{only108}
The real sextic space curve determined by $c(x,y,z)=0$ and \linebreak $x^2+y^2+z^2=25$ has 108 totally-real tritangent planes.
\end{theorem}
\begin{proof}
We break the proof into parts based on the type of tritangent plane.  There are tritangents which touch three distinct ovals, tritangents which touch an oval twice and another oval once, and tritangents which touch a single oval three times.  We label them $(1,1,1)$, $(2,1)$, and $(3)$\textnormal{-}tritangents respectively.

\begin{description}
\item[80 $(1,1,1)\textnormal{-}$tritangents.] %
Given any three ovals of the curve, there exist $2^3 =8$ classes of planes which separate them.  That is, a given plane has some of the ovals ``above'' it, and some ``below''.
Such a plane can be moved so that it touches the three ovals each at one point in a unique way (For an analogue of this in the plane, consider two general non-concentric ellipses and their four bitangents.)
This yields $8\binom{5}{3}=80$ tritangents, and there are clearly no other tritangent planes meeting each of three ovals once.

\item[12+18 $(2,1)\textnormal{-}$tritangents.]
For each $O_i$, there are four tritangents which touch it twice.  To find them,  
consider the plane tangent to $O_i$ at
its northernmost point and southernmost point.  
The plane can be rotated so that it keeps two tangency points to $O_i$.  As it rotates, it meets the other two $O_j$ each once yielding two tritangents for a total of $3\cdot 2=6$ such tritangents.  The projection of these two tritangents to the $xy$-plane is pictured in Figure~\ref{fig:ovalsprojected}.

\begin{figure}
\centering
\includegraphics[width=0.52\linewidth,trim=0mm 0mm 0mm 0mm,clip]{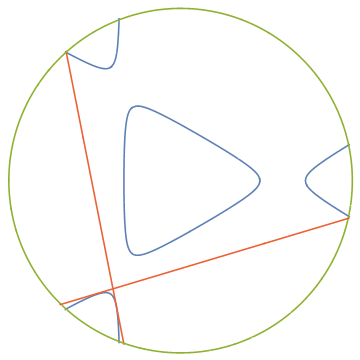}
\caption{Two $(2,1)$-tritangents projected to $xy$-plane.}
\label{fig:ovalsprojected}
\end{figure}

Similarly, for $N$ (resp. $S$), there are nine tritangents which touch it twice.  To see them, pick a side of the triangle of $N$ (resp. $S$), and consider the opposing $O_i$.  There is a tritangent which touches $N$ (resp. $S$) at two points along this side and touches $O_i$ at its northernmost point, and similarly one which touches the $O_i$'s southernmost point.  Finally, there is a tritangent which touches the opposing point of $S$ (resp. $N$).  This yields $9\cdot 2 = 18$ tritangents.

Clearly there are no additional $(2,1)$-tritangents meeting $N,S$ twice.  
We now show there are no additional $(2,1)$-tritangents meeting $O_i$ twice.
Observe that the oval $O_i$ has two reflectional symmmetries, one through the ``equator'' and one through the great circle determined by the northernmost and southernmost points.  If $p \in O_i$ is a point which is not fixed by either reflection, then the images under reflection -- denoted $p'$ and $p''$ each share a tangent plane to $O_i$ with $p$, that is, the tangent line $T_p O_i$ to $O_i$ at $p$ intersects $T_{p'} O_i$ and $T_{p''} O_i$.  It is tedious but not difficult to check that the two planes determined by these lines are the only bitangent planes to $O_i$ at $p$.
If $p'$ is given by reflecting $p$ through the $xy$-plane, then the corresponding bitangent is a tritangent only if the projection (Figure~\ref{fig:ovalsprojected}) is a bitangent line, and it is apparent from the picture that we have already claimed all these.  If $p''$ is the other reflection, then the bitangent to $p$ and $p''$ cuts out a circle on the sphere which is contained in $O_i$.  Therefore it cannot be tangent at a point on another oval.

\item[4 $(3)\textnormal{-}$tritangents.] %
$N$ has three maxima with respect to height in the $z$-direction. 
There is a plane which touches the oval at these three points. Similarly, it has three minima, and there is another tritangent plane there.  The same is true for $S$. We thus have $2\cdot 2=4$ four more tritangent planes.

It is easy to see that there are no more $(3)$-tritangent planes meeting $N$ or $S$ only. A tritangent plane also cannot meet $O_i$ only as, by symmetry, such a plane would have to touch $O_i$ at either its northernmost or southernmost point, but there are not other points sharing a bitangent with either of these.
\qed
\end{description}
\end{proof}

We have shown that Emch's curve has fewer totally-real tritangents than was previously thought.  A natural question is thus reopened.
\begin{question}
Does there exist a canonically embedded real space sextic with $120$ totally-real tritangent planes? 
\end{question}
 

\begin{figure}
\centering
\begin{minipage}{.48\textwidth}
\label{fig:cubicIntersectSphere}
\centering
\includegraphics[width=0.77\linewidth,trim = 20mm 19mm 15mm 15mm, clip]{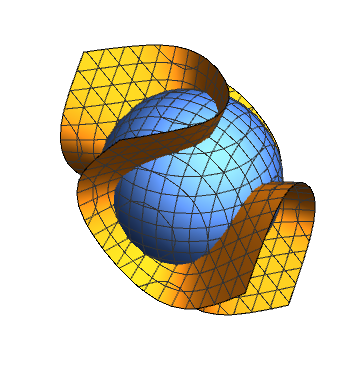}
\caption{The intersection of a cubic and a quadric yields a sextic.}
\end{minipage}
\hfill
\begin{minipage}{.48\textwidth}
\centering
\includegraphics[width=0.8\linewidth,trim = 27mm 32mm 27mm 37mm, clip]{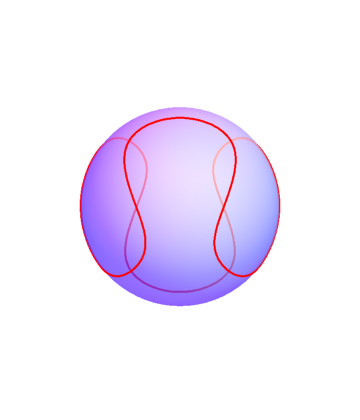}
\caption{A real sextic curve with a single connected component.}
\label{fig:connRealSextic}
\end{minipage}
\end{figure}

We now consider a curve with significantly fewer totally-real tritangent planes. Let $C$ be the sextic determined by the intersection of the sphere $S_2$ defined by 
\[ x^2 + y^2 + z^2 = 1 \]
with the ``Clebsch diagonal cubic'' $S_3$ defined by
\begin{align*} & 81(x^3 + y^3 + z^3) - 
  189(x^2y + x^2z + xy^2 + y^2z + xz^2 + yz^2) + 54xyz \\
  & + 
  126(xy + xz + yz) - 9(x^2 + y^2 + z^2) - 9(x + y + z) + 1=0.
  \end{align*}
This cubic surface has a threefold rotational symmetry about the axis $x=y=z$.
In Figure~\ref{fig:connRealSextic}, this corresponds to the $\frac{2\pi}{3}$ rotation about the ``north/south poles'', which we will call $p_N$ and $p_S$. Proposition~\ref{realTheta} implies that since $C$ has real points, it has eight real tritangent planes. As the following theorem shows, these planes exist and are all totally-real. 

\begin{theorem}
The curve $C$ has exactly 8 totally-real tritangent planes.
\end{theorem}
\begin{proof}
Let $q$ denote a point on $C$ which minimizes the distance to $p_N$. 
This point lies on a circle in $\RR^3$ which is the intersection of $S_2$ with a sphere centered at $p_N$ of radius $d(p_N, q)$.  
By the threefold rotational symmetry, there are at least three distinct points at which $C$ 
touches the circle.  Thus, the plane containing this circle is tritangent to $C$ and there are exactly three tangency points.  The same argument for $p_S$ gives a second tritangent plane to $C$.

Observe that $C$ has a reflectional symmetry through the plane determined by $p_N$, $p_S$, and $q$.
The three points associated to any tritangent to $C$ must lie on a circle on $S_2$ which cannot cross $C$.  By the reflectional symmetry of $C$, this circle must meet $C$ at one of the points $q$ from the two known tritangents. 
Such a circle, of a sufficiently small radius, meets $C$ at no other points.  With a sufficiently large radius, it meets $C$ transversely at multiple points.  Thus, there are circles on either side of $C$, which touch but don't cross the curve at some point other than $q$.  Since the total intersection number of the two curves cannot exceed $3 \cdot 2 = 6$, and by the reflectional symmetry, this circle is tangent at exactly three points.
Hence, we find one additional circle for each of the six points $q$, for a total of 8 tritangents planes.
By Proposition~\ref{thm:tritangentsAlgebraic}, this is the maximum possible.
\qed
\end{proof}

\section{Tropical space sextics}
\label{sec:trop}
In this section, we consider tropical space sextics of genus $4$, and show that they satisfy enumerative properties that are analogous to the ones presented above. 
We begin with a brief overview of several topics regarding tropical curves. We mostly focus on notions that are necessary for defining and studying tropical tritangent planes. The interested reader may find a more thorough treatment in \cite{MS15}.

\begin{definition}
A  tropical curve is a metric  graph $\Gamma$ embedded in $\mathbb{R}^n$, together with an integer weight function on the edges, such that
\begin{itemize}
\item The direction vector of each edge is rational.
\item At each vertex, the weighted sum of the primitive integral vectors of the edges
around the vertex is zero.
\end{itemize}
The \emph{genus} of a tropical curve is the first Betti number $\dim H^1(\Gamma,\mathbb{Z})$ of the graph. 
\end{definition}
We  assume throughout that the weights on the edges are all one.

\begin{definition}
A tropical curve is of \emph{degree} $d$ if it has $d$ infinite ends in each of the directions $-e_1,-e_2,\ldots,-e_n, e_1+e_2+\cdots+e_n$. 
A plane curve is of \emph{bi-degree} $(d_1,d_2)$ if it has $d_1$ ends in each of the directions $e_2,-e_2$, and $d_2$ ends in the directions $e_1,-e_1$.
\end{definition}


\begin{example}
\label{Ex:elliptic}
The graph in Figure~\ref{fig:elliptic} is a tropical plane curve of degree $3$ and genus~$1$. 

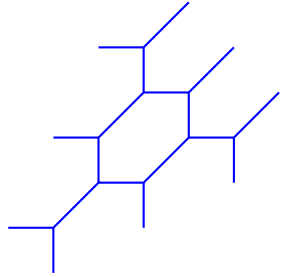
\begin{figure}[h]
\centering
\begin{tikzpicture}[scale=.6]
\draw[thick, blue] (0,0)--(1,0)--(2,1)--(2,2)--(1,2)--(0,1)--(0,0);
\draw[thick, blue] (0,0) to (-1,-1);
\draw[thick, blue] (-1,-1) to (-2,-1);
\draw[thick, blue] (-1,-1) to (-1,-2);
\draw[thick, blue] (0,1) to (-1,1);
\draw[thick, blue] (1,2) to (1,3);
\draw[thick, blue] (1,3) to (0,3);
\draw[thick, blue] (1,3) to (2,4);
\draw[thick, blue] (2,2) to (3,3);
\draw[thick, blue] (2,1) to (3,1);
\draw[thick, blue] (3,1) to (4,2);
\draw[thick, blue] (3,1) to (3,0);
\draw[thick, blue] (1,0) to (1,-1);



\end{tikzpicture}
\caption{A tropical elliptic curve.}
\label{fig:elliptic}
\end{figure}
\end{example}

\begin{definition}
A \emph{tropical   plane} in $\mathbb{R}^3$ is a two dimensional polyhedral complex with a unique vertex $v$, whose $1$-skeleton consists of the rays\linebreak $v+\mathbb{R}_{\geq 0} (-e_1), v+\mathbb{R}_{\geq 0} (-e_2), v+\mathbb{R}_{\geq 0} (-e_3)$, and  $v+\mathbb{R}_{\geq 0} (e_1+e_2+e_3)$. The maximal faces are the cones generated by each pair of rays. In other words, it is a translation of the $2$-skeleton of the fan of the toric variety $\mathbb{P}^3$.
\end{definition}

More generally, a \emph{tropical variety} is a balanced polyhedral complex in $\mathbb{R}^n$, in the sense of \cite[Definition 3.3.1]{MS15}. Tropical hypersurfaces (namely tropical varieties of codimension $1$) are simply constructed by taking the dual complex of a subdivision of a polytope with integer vertices. The hypersurface is \emph{tropically smooth} if the subdivision is a unimodular triangulation. 

For the rest of this section, we assume that $\Bbbk$ is an algebraically closed field, endowed with a non-trivial non-archimedean valuation $\nu$. 
For simplicity, we may choose $\Bbbk$ to be the field of Puiseux series over $\mathbb{C}$, consisting of all elements of the form
\[
x = a_k\cdot t^{\frac{k}{n}} + a_{k+1}\cdot t^{\frac{k+1}{n}} +\cdots
\]
for all choices of $k\in\mathbb{Z}, n\in\mathbb{N}$, and  coefficients $a_i\in\mathbb{C}$. The valuation is given by $\nu(x)=\frac{k}{n}$.

Let  $X$ be a variety  in $(\Bbbk^*)^n$. The \emph{tropicalization} map $\text{trop}:X(K)\to\mathbb{R}^n$ is defined by
\[
\text{trop}(x_1,x_2,\ldots,x_n) = (-\nu(x_1),-\nu(x_2),\ldots,-\nu(x_n)).
\]
The tropicalization of $X$ is the closure in $\mathbb{R}^n$ of $\text{trop}(X(\Bbbk))$.

The reader will be relieved to know that the tropicalization of a variety is, indeed, a tropical variety of the same dimension. The tropicalization of a generic curve of degree $d$ (resp. bi-degree $(d_1,d_2)$) is a tropical curve of degree $d$ (resp. bi-degree $(d_1,d_2)$). Similarly, the tropicalization of a plane in $(\Bbbk^*)^3$ is a tropical plane in $\mathbb{R}^3$. 

\begin{example}
The tropicalization of the degree $3$ plane curve
\[
f = t + x + y + xy + t\cdot x^2 + t\cdot y^2 + t^2\cdot x^2y + t^2\cdot xy^2 + t^4\cdot x^3 + t^4\cdot y^3
\]
is the tropical curve of degree $3$ appearing in Figure~\ref{fig:elliptic}.
The tropicalization of the curve of bi-degree $(1,2)$
\[
g = 1 + x + y + t\cdot xy + t^3\cdot xy^2 + t^3\cdot y^2
\]
is the tropical curve of bi-degree $(1,2)$ depicted in Figure~\ref{fig:oneTwoCurve}.
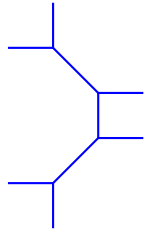
\begin{figure}
\centering
\begin{tikzpicture}[scale=.6]
\draw[thick, blue] (-1,0)--(0,0)--(1,1)--(1,2)--(0,3)--(-1,3);
\draw[thick, blue] (0,0)--(0,-1);
\draw[thick, blue] (0,3)--(0,4);
\draw[thick, blue] (1,1)--(2,1);
\draw[thick, blue] (1,2)--(2,2);


\end{tikzpicture}
\caption{A tropical $\mathbb{P}^1\times\mathbb{P}^1$ curve of bi-degree (1,2).}
\label{fig:oneTwoCurve}
\end{figure}

\end{example}

A common difficulty in tropical geometry is that occasionally, tropical varieties intersect at higher than expected dimension. This can happen even when the tropical varieties in question are tropicalizations of algebraic varieties that do intersect in the expected dimension. One way to address this issue is to  ``force'' the varieties to intersect properly, using the so-called  \emph{stable intersection}. Colloquially, this is done by generically perturbing the tropical varieties, and taking the limit of their intersection as the perturbations tend to zero. 

More precisely, whenever cells $\sigma_1,\sigma_2$ of tropical varieties $\Sigma_1,\Sigma_2$ span  $\mathbb{R}^n$, their set-theoretic intersection will be a cell in the stable intersection. To assign a weight to this cell, consider the lattices $N_1,N_2$ obtained by intersection $\sigma_1,\sigma_2$ with the lattice $\mathbb{Z}^n$ in $\mathbb{R}^n$. The weight given to the cell is 
\[
m(\sigma_1)\cdot m(\sigma_2)\cdot [\mathbb{Z}^n:N_1+N_2],
\]
where $m(\sigma_i)$ is the weight of each cell $\sigma_i$. 


\begin{example}\label{ex:stableIntersection}
Consider the two curves depicted in Figure~\ref{fig:intersectionExample}. A basis for the ray of the blue curve at the intersection point is given by the primitive vector $(1,1)$, whereas a basis for the ray of the red curve is given by $(1,-1)$. The multiplicity of the point of intersection is therefore 
\[
\det\begin{pmatrix}
  1 & 1 \\
  1 & -1 \\
 \end{pmatrix} = 2.
 \]
This is consistent with tropical Bezout's theorem, since this is an intersection of a line with a $(1,1)$-curve.
 
\begin{figure}
\begin{minipage}{0.48\textwidth}
\centering
\begin{tikzpicture}[scale=0.8]
\draw[blue] (-1,0)--(0,0)--(0,-1);
\draw[blue] (0,0)--(2,2);

\draw[red] (-1,1)--(0,1)--(0,2);
\draw[red] (0,1)--(1,0);
\draw[red] (1,-1)--(1,0)--(2,0);

\draw[fill,purple] (.5,.5) circle [radius=0.1];

\end{tikzpicture}
\caption{Two tropical curves meeting with multiplicity $2$.}
\label{fig:intersectionExample}
\end{minipage} %
\hfill %
\begin{minipage}{0.48\textwidth}
\centering
\begin{tikzpicture}[scale=0.8]
\draw (0,1.5);
\draw[blue] (-1,-1)--(0,0)--(0,1);
\draw[blue] (0,0)--(2,0);
\draw[blue] (3,-1)--(2,0)--(2,1);

\draw[red] (-1,0.07)--(3,0.07);
\draw[fill,purple] (0,0) circle [radius=0.1];
\draw[fill,purple] (2,0) circle [radius=0.1];

\end{tikzpicture}
\vspace{0.4cm}
\caption{Tropical curves intersecting stably at two points.}
\label{stupid}
\end{minipage}
\end{figure}

The two curves in Figure~\ref{stupid} don't  intersect properly. However, by slightly perturbing the red curve, the two curves intersect at two points with multiplicity $1$. Taking the limit as the red curve returns to its original position, we see that the stable intersection has multiplicity $2$. 

\end{example}

\begin{definition}
Two tropical varieties are \emph{tangent} at a point $q$ if their intersection at $q$ has weight at least $2$, or if $q$ is in the interior of a bounded segment of their set-theoretic intersection. They are tritangent to each other if they are tangent at three disjoint places (either points or segments) counted with multiplicity. Two tritangents are \emph{equivalent} if the tangency points are linearly equivalent divisors.
\end{definition}
In particular, tritangent varieties may meet at three places with multiplicity $2$ each, at two places with multiplicity $4$ and $2$, or at one place with multiplicity $6$.  See Figure~\ref{fig:15tritangents} for various examples of curves that are tritangent to each other.

When $\Gamma$ is the tropicalization of a curve $C$, then any tangent of $C$ tropicalizes to a tangent of $\Gamma$, and in particular a tritangent tropicalizes to a tritangent. This can be deduced, for example, from \cite[Theorem 6.4]{OR11}.  

\section{Tropical divisors, theta characteristics and tritangent planes}
Divisors on tropical curves are defined analogously to algebraic curves. A \emph{divisor} $D$ on $\Gamma$ is a finite formal sum 
\[
D = a_1 p_1 + a_2 p_2 +\cdots+a_k p_k,
\]
where each $a_i$ is an integer, and each $p_i$ is a point of the curve. The \emph{degree} of $D$ is $a_1+a_2+\cdots+a_k$, and we say that $D$ has $a_i$ \emph{chips} at $p_i$. The divisor is said to be \emph{effective} if $a_i\geq 0$ for every $i$. 
In analogy with the algebraic case, there is a a suitable equivalence relation between divisors, and a notion of rank which, roughly speaking, reflects the dimension in which the divisor moves. The curve has a canonical divisor class   $K_\Gamma$ which fits in a tropical Riemann--Roch theorem
\[
r(D)-r(K_\Gamma-D) = \deg(D)-g+1,
\]
where $r$ is the rank of a divisor, and $g$ the genus of the tropical curve (\cite[Theorem 1.12]{BN07}).
See \cite{BJ16} for a lucid introduction to tropical divisor theory. 

Now, a \emph{tropical theta characteristic} on a tropical curve $\Gamma$ is defined in exactly the same way as algebraic theta characteristic.  Namely, it is a divisor class $[D]$ such that $2D\simeq K_\Gamma$. Since the Jacobian of a tropical curve of genus $g$ is isomorphic to a $g$-dimensional real torus $\mathbb{R}^g/\mathbb{Z}^g$ (\cite[Theorem 3.4]{BF11} and \cite[Section 5]{BBC17}), and theta characteristics are in bijection with its $2$-torsion points, there are $2^g$ theta characteristics. One of them is non-effective and the rest are effective. They are easily computed via the following algorithm, introduced by Zharkov \cite{Zha10}.

To get an effective theta characteristic, fix a cycle $\gamma$ in $\Gamma$. At every point $p$ that locally maximizes the distance from $\gamma$, place $a-1$ chips at $p$, where  $a$ is the number of incoming edges at $p$ from the direction of $\gamma$. The process is often described pictorially as follows. A fire spreads along the graph at equal speed away from $\gamma$. If $a$ is the number of incoming  fires at a point $p$,  we place $a-1$ chips at that point. To obtain the unique non-effective theta characteristic, repeat the same process, but replace $\gamma$ with the set of vertices of the graph, and place a negative chip at each vertex.

\begin{example}
Let $\Gamma$ be the curve in Figure~\ref{fig:theta} (where the infinite ends are omitted). As the genus is $2$, we expect the theta characteristics   to have degree $1$. For the picture on the left,  the middle of the bottom horizontal edge is the unique local maximum from the chosen cycle (marked with a red cycle). The corresponding theta characteristic has a single chip at that point. For the picture on the right, the middle of each horizontal edge locally maximizes the  distance from the vertices. The non-effective theta characteristic of this curve therefore consists of a negative chip at each of the three vertices, and a chip at the middle of each horizontal edge. Each of these divisors, when multiplied by two is equivalent to the canonical divisor, so they are half canonical.

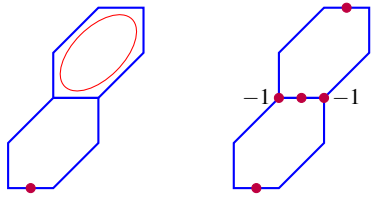
\begin{figure}[h]
\centering
\begin{tikzpicture}[scale=.6]
\draw[thick, blue] (0,0)--(1,0)--(2,1)--(2,2)--(1,2)--(0,1)--(0,0);
\draw[thick, blue] (0,0) -- (-1,-1)--(-1,-2)--(0,-2)--(1,-1)--(1,0)--(0,0);
\draw[fill,purple] (-0.5,-2) circle [radius=0.1];
\draw[rotate around={135:(1,1)},red] (1,1) ellipse (16pt and 30pt);

\begin{scope}[shift={(5,0)}]
\draw[thick, blue] (0,0)--(1,0)--(2,1)--(2,2)--(1,2)--(0,1)--(0,0);
\draw[thick, blue] (0,0) -- (-1,-1)--(-1,-2)--(0,-2)--(1,-1)--(1,0)--(0,0);
\draw[fill,purple] (-0.5,-2) circle [radius=0.1];
\draw[fill,purple] (.5,0) circle [radius=0.1];
\draw[fill,purple] (1.5,2) circle [radius=0.1];

\draw[fill,purple] (0,0) circle [radius=0.1];
\node [left] at (0,0) {$-1$};
\draw[fill,purple] (1,0) circle [radius=0.1];
\node [right] at (1,0) {$-1$};

\end{scope}
\end{tikzpicture}
\caption{Two theta characteristics on a curve of genus $2$.}
\label{fig:theta}
\end{figure}
\end{example}

Let $\Gamma$ be a tropical curve of degree $6$ and genus $4$ in $\mathbb{R}^3$. Its stable intersection with a tropical plane $\Pi$ is an effective divisor of degree $6$. We claim that its rank is $3$. Indeed, we can find a tropical plane through any three general points, and any pair of divisors obtained this way is linearly equivalent. 
By the tropical Riemann--Roch theorem,  a divisor of degree $6$ and rank $3$ has to be equivalent to the canonical divisor. 
Consequently,  every tritangent plane gives rise to an effective theta characteristic on $\Gamma$. By definition, non-equivalent tritangent planes correspond to different theta characteristics. 
It follows that the number of equivalence classes of tritangent planes is bounded above by the number of effective theta characteristics which is  $2^4-1 = 15$.  


\begin{theorem}\label{thm:mainTropThm}
A smooth tropical sextic curve $\Gamma$ in $\mathbb{R}^3$ has at most $15$ classes of tritangent planes. If it is the  tropicalization of a sextic $C$ on a smooth quadric in $\mathbb{P}^3$, then it has  exactly $15$ equivalence classes of tritangent planes. 
\end{theorem}
\begin{proof}
The first statement was shown in the discussion above.
By Theorem~\ref{thm:tritangentsAlgebraic}, the curve $C$ has $120$ tritangent planes. If $H$ is tritangent to $C$, then $\trop(H)$ is tritangent to $\Gamma$. 
Moreover, by \cite[Theorem 1.1]{JL16}, each tropical effective theta characteristic of $\Gamma$ is the tropicalization of $8$ odd theta characteristics of $C$. Since different theta characteristics correspond to non-equivalent tritangent planes, there are $15$ distinct classes of planes tritangent to $\Gamma$. 
\qed
\end{proof}

\begin{remark}\label{rmk:infinitePlanes}
It is quite possible for a tropical sextic to have an infinite continuous family of tritangent planes. However,  the tangency points of such a family will consist of linearly equivalent divisors, and as such the corresponding tritangent planes are equivalent. If the sextic is the tropicalization of an algebraic sextic, then each equivalence  consists of $8$ tritangent planes  (counted with multiplicity) that can be lifted to tritangent planes of the algebraic sextic.
\end{remark}

The proof above relies on the fact that the given tropical curve arises as the tropicalization of an algebraic curve. However, we expect the result to be true in general. 
\begin{conjecture}\label{conj:tritangents}
Every tropical sextic curve of genus $4$ in $\mathbb{R}^3$ has exactly $15$ equivalence classes of tritangent planes. 
\end{conjecture}

We now wish to explore the tropical analogue of the relation between quadric surfaces and $\mathbb{P}^1\times\mathbb{P}^1$, as described in Remark~\ref{rmk:oneOneCurves}. To examine the analogous statement in tropical geometry, recall that a smooth  tropical quadric in $\mathbb{R}^3$ is dual to a unimodular triangulation of the $3$-simplex with vertices $(0,0,0), (2,0,0), (0,2,0), (0,0,2)$. By the proof of \cite[Theorem 4.5.8]{MS15}, such a triangulation has a unique interior edge, corresponding to a unique bounded face of the quadric. This face can be seen as a model for tropical $\mathbb{P}^1\times\mathbb{P}^1$. More precisely, 

\begin{lemma}\label{lemma:tropicalRuled}
Let $\Sigma$ be a tropical smooth quadric surface in $\mathbb{R}^3$, and fix a rectangle $R$ in $\mathbb{R}^2$. Then there is an affine linear map from $R$ onto a parallelogram in the bounded face of $\Sigma$ inducing a bijection between curves of bi-degree $(d,d)$ in $\mathbb{R}^2$ whose bounded edges are all contained in $R$, and curves of degree $2d$ in $\mathbb{R}^3$ that are contained in $\Sigma$, and their bounded edges are contained in the parallelogram. 
\end{lemma}
\begin{proof} 
Denote by $F$ the unique bounded face of $\Sigma$. 
By \cite[Theorem 4.5.8]{MS15}, there are two tropical lines through each point of $F$ that are fully contained in $\Sigma$. The bounded edge of each is fully contained in $F$, and is parallel to one of two directions, which we denote $u_1$ and $u_2$. These directions are determined by $F$, and do not depend on the point. 
Every ray in $F$  that is parallel to $u_i$ (for $i=1,2$) can be extended past the boundary of $F$ by attaching  infinite ends in two of the directions $\{-e_1,-e_2,-e_3,e_0\}$ (where $e_0=e_1+e_2+e_3$), and a ray parallel to $-u_i$  can be extended by attaching ends in the two remaining directions. 

Let $\varphi:R\to F$ be the affine linear map that sends a vertex of $R$  to $p$, and the two adjacent vertices to $p+\lambda u_1, p+\lambda u_2$ for some fixed point $p$ of $F$, and a small enough $\lambda$ so that the image is contained in $F$. Let $\Gamma$ be a $(d,d)$-curve $\Gamma$ in $\mathbb{R}^2$ whose bounded edges are contained in $R$. Each of the infinite ends is mapped by $\varphi$ to rays that are parallel to $\pm u_1$ or $\pm u_2$. By extending all the rays emanating from $\varphi(\Gamma\cap R)$ according to the discussion above, we get a curve of degree $2d$ in $\mathbb{R}^3$ that is contained in $\Sigma$. 
\qed
\end{proof}

The lemma suggests an alternative approach for proving Conjecture~\ref{conj:tritangents}. 
\begin{corollary}\label{cor:ruledQuadric}
For tropical sextics whose bounded edges are contained in the bounded face of a smooth tropical quadric, Conjecture~\ref{conj:tritangents} is equivalent to the statement that every $(3,3)$-curve in $\mathbb{R}^2$ has $15$ classes of tritangent $(1,1)$\nobreakdash-curves.
\end{corollary}
\begin{proof}
Let $\Gamma$ be a $(3,3)$-curve in $\mathbb{R}^2$, mapping to a sextic in $\mathbb{R}^3$ by $\varphi$ as in Lemma~\ref{lemma:tropicalRuled}. 
Every $(1,1)$-curve that is tritangent to $\Gamma$ maps to a conic curve in $\mathbb{R}^3$ that is tritangent to $\varphi(\Gamma)$ and contained in $\Sigma$. 
By construction, the conic curve is not contained in a tropical line or in any of the standard planes in $\mathbb{R}^3$. Therefore, we can find 3 general points on it that span a unique tropical plane. This plane contains the conic curve and is tritangent to $\varphi(\Gamma)$. 
\qed
\end{proof}

\begin{example}\label{ex:15tritangents}
The Figure below shows  $15$ equivalence classes of tritangent $(1,1)$-curves to a tropical $(3,3)$-curve in $\mathbb{R}^2$. By Corollary~\ref{cor:ruledQuadric}, this curve corresponds to a tropical sextic in $\mathbb{R}^3$ reaching the maximal number of tritangent planes. 
To find each tritangent curve,  we choose a non-trivial cycle in the graph, compute the corresponding theta characteristic via Zharkov's algorithm, and find a $(1,1)$-curve through it. 

\begin{figure}[!htb]
\centering
\input{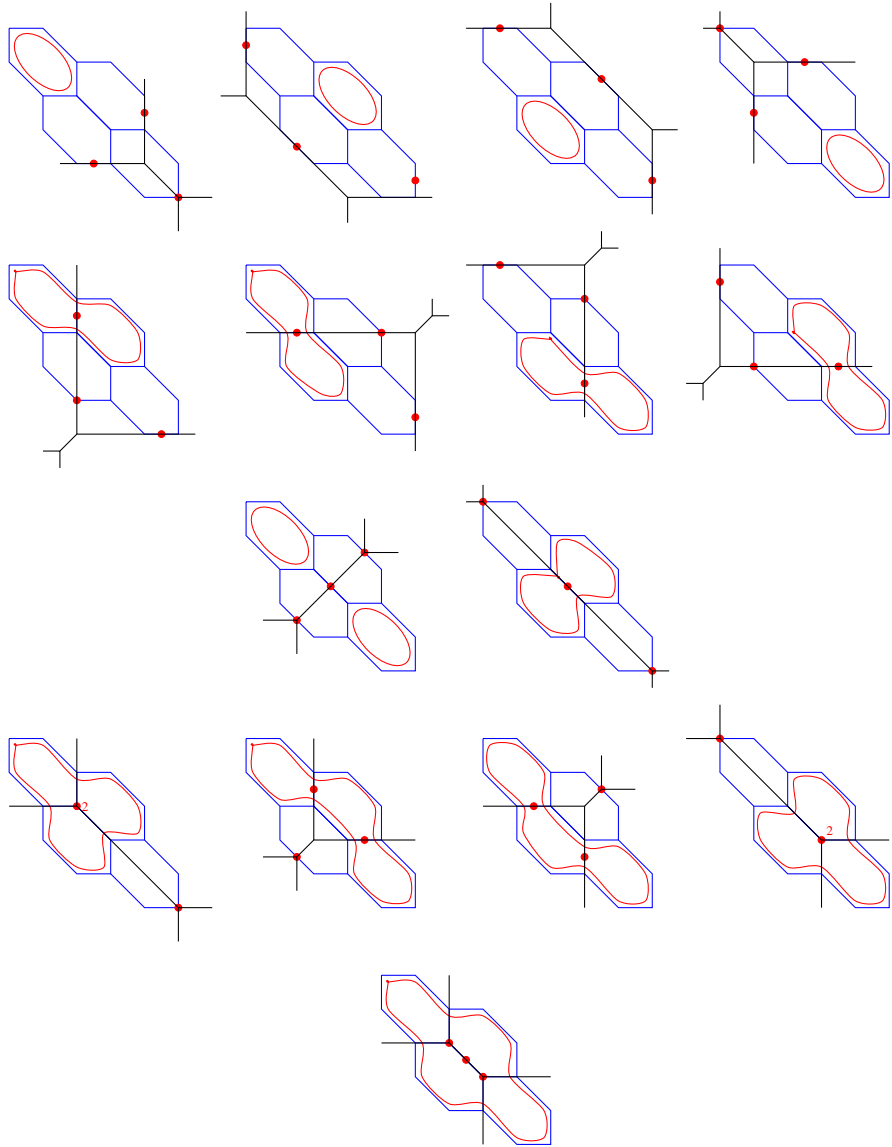}
\caption{15 tritangents to a tropical sextic.}
\label{fig:15tritangents}
\end{figure}

We stress that some of the odd theta characteristics, in fact, give rise to infinitely many tritangent $(1,1)$-curves, for instance the third tritangent in the second row in Figure~\ref{fig:15tritangents} when counting from the top left. However, as the tangency points of these different $(1,1)$-curves are equivalent divisors, they are all in the same equivalence class (cf.  Remark~\ref{rmk:infinitePlanes}).

\end{example}

\begin{acknowledgement}
This article was initiated during the Apprenticeship Weeks (22 August-2 September 2016), led by Bernd Sturmfels, as part of the Combinatorial Algebraic Geometry Semester at the Fields Institute.
We thank Lars Kastner, Sara Lambolgia, Steffen Marcus, Kalina Mincheva, Evan Nash, and Bach Tran who participated with us in the morning and evening sessions that led to this project. A great thanks goes to Kristin Shaw and Frank Sottile for many helpful discussions and insightful suggestions. We thank the editors Greg Smith, Bernd Sturmfels, and the liaison committee for always being vigilant, and making sure that this book sees the light of day. Finally, we thank the anonymous referees for their valuable comments and suggestions.
\end{acknowledgement}

\end{document}